\documentclass[11pt]{amsart}

\usepackage{epigamath}

\usepackage[english]{babel}

\usepackage{tikz-cd}
\usepackage{enumitem}

\theoremstyle{plain}
\newtheorem{theorem}[subsection]{Theorem}
\newtheorem{lemma}[subsection]{Lemma}
\newtheorem{conjecture}[subsection]{Conjecture}
\newtheorem{corollary}[subsection]{Corollary}
\newtheorem{proposition}[subsection]{Proposition}

\theoremstyle{definition}

\newtheorem{remark}[subsection]{Remark}

\renewcommand{\leq}{\leqslant}

\newcommand{\CC}{\ensuremath{\mathbb{C}}}

\newcommand{\QQ}{\ensuremath{\mathbb{Q}}}

\newcommand{\ZZ}{\ensuremath{\mathbb{Z}}}

\DeclareMathOperator{\h}{\mathfrak{h}} 
 
\DeclareMathOperator{\Km}{\mathrm{Km}} 
 
\DeclareMathOperator{\tr}{\mathrm{tr}} 

\EpigaVolumeYear{7}{2023} \EpigaArticleNr{4} \ReceivedOn{July 4, 2022}
\InFinalFormOn{September 28, 2022}
\AcceptedOn{October 22, 2022}

\title{On the motive of O'Grady's six-dimensional \\ hyper-K\"ahler varieties}
\titlemark{Motives of OG6 varieties}

\author{Salvatore Floccari}
\address{Institute of Algebraic Geometry, Leibniz University Hannover, Germany}
\email{floccari@math.uni-hannover.de}

\authormark{S. Floccari}

\AbstractInEnglish{We prove that the rational Chow motive of a six-dimensional hyper-K\"{a}hler variety obtained as symplectic resolution of O'Grady type of a singular moduli space of semistable sheaves on an Abelian surface $A$ belongs to the tensor category of motives generated by the motive of $A$. We in fact give a formula for the rational Chow motive of such a variety in terms of that of the surface. As a consequence, the conjectures of Hodge and Tate hold for many hyper-K\"{a}hler varieties of OG6-type.}

\MSCclass{14C15, 14C30, 14D20, 14J42}

\KeyWords{Hyper-K\"ahler varieties, moduli spaces, motives, Hodge conjecture}

\begin{document}



\maketitle

\begin{prelims}

\DisplayAbstractInEnglish

\bigskip

\DisplayKeyWords

\medskip

\DisplayMSCclass







\end{prelims}


\newpage

\setcounter{tocdepth}{1}

\tableofcontents


        \section{Introduction}
	
	Moduli spaces of stable sheaves on a K3 or Abelian surface constitute a prominent source of construction of hyper-K\"{a}hler varieties. The geometry of the higher dimensional varieties constructed in this way can be fruitfully studied  through the relationship with the underlying surface. 
	For instance, the rational Hodge structure on the cohomology of the hyper-K\"{a}hler varieties so obtained can be expressed in terms of that on the cohomology of the surface (\emph{cf.}~\cite{green2019llv}). When such a relation can be established at the level of motives, it has important consequences for the study of algebraic cycles.
	
	Let us now be more precise on the moduli spaces and hyper-K\"{a}hler varieties we will consider. Let $A$ be a complex Abelian surface. Its Mukai lattice is 	
	\[
	\widetilde{H}^2(A,\ZZ) \coloneqq H^0(A,\ZZ)\oplus H^{2}(A,\ZZ) \oplus H^4(A,\ZZ),
	\]
	with the pairing $\bigl((a,b,c), (a',b',c')\bigr)= (b,b')-ac'-a'c$. It is equipped with a weight~$2$ Hodge structure induced by that on $H^2(A,\ZZ)$ and declaring $H^0(A,\ZZ)$ and $H^4(A,\ZZ)$ of type $(1,1)$. To any coherent sheaf $\mathcal{F}$ on $A$ is assigned its Mukai vector $v(\mathcal{F})\coloneqq \mathrm{ch}(\mathcal{F})$;
	a vector $v\in \widetilde{H}^2(A,\ZZ)$ is called effective if $v=v(\mathcal{F})$ for some coherent sheaf $\mathcal{F}$.
	
	Given an effective Mukai vector $v\in \widetilde{H}^2(A,\ZZ)$ and a $v$-generic polarization $H$, there exists a projective moduli space $M_v(A,H)$ of (S-equivalence classes of) $H$-semistable coherent sheaves on $A$ with Mukai vector $v$, see \cite{huybrechts2010geometry}. It is irreducible of dimension $(v,v)+2$, and its smooth locus admits a nowhere degenerate holomorphic $2$-form~\cite{Muk84}.
	By \cite{Yos01}, the Albanese map $M_{v}(A,H)\to A \times \hat{A}$ is an isotrivial fibration; we denote by~$K_v(A,H)$ the fibre of this morphism.
	
	When the effective Mukai vector $w$ is primitive, the moduli space $M_w(A,H)$ is a smooth and projective variety, and $K_w(A, H)$ is a hyper-K\"{a}hler variety of generalized Kummer type (\emph{cf.}~\cite{Yos01}).
	 
	When instead $v$ is not primitive the moduli space $M_v(A,H)$ is singular. O'Grady discovered \cite{O'G99} that $M_{(2,0,-2)}(A,H)$ admits a crepant resolution $\widetilde{M}_{(2,0,-2)}(A,H)$, by which we mean that the pull-back of the holomorphic $2$-form on the regular locus of $M_{(2,0-2)}(A,H)$ extends to a nowhere degenerate holomorphic $2$-form on the resolution. By \cite{LS06} and \cite{KLS06}, a crepant resolution $\widetilde{M}_v(A,H)\to M_v(A,H)$ exists if and only if $v=2v_0$ with $v_0$ a primitive and effective Mukai vector of square $2$. In this case we obtain as well a crepant resolution $\widetilde{K}_v(A,H)$ of $K_v(A,H)$; by \cite{O'G03} and \cite{PR13}, the $\widetilde{K}_v(A,H)$ are hyper-K\"{a}hler varieties of dimension $6$, all deformation equivalent to each other. They are called of OG$6$-type.
	
	Each of the varieties introduced above is expected to be motivated by the surface; this means that the rational Chow motive of such a variety should belong to the tensor subcategory generated by the motive of~$A$.
	For the smooth and projective moduli spaces $M_w(A,H)$ associated to a primitive Mukai vector $w$, this has been confirmed by B\"ulles~\cite{Bue18}. When $v=2v_0$ with $v_0$ primitive, effective and of square~$2$, an extension of his argument allows to show that also the $10$-dimensional varieties~$\widetilde{M}_v(A,H)$ are motivated by $A$, see \cite{FFZ}. Both results are based on Markman's work \cite{Mar02}.
	For the hyper-K\"{a}hler varieties $K_w(A,H)$ and $\widetilde{K}_v(A,H)$, however, the question remained open; in this note we deal with the OG$6$-type varieties $\widetilde{K}_v(A,H)$.

	\begin{theorem}\label{thm:main}
		Let $A$ be a complex Abelian surface, $v=2v_0$ a Mukai vector with~$v_0$ primitive, effective and of square $2$, and let $H$ be a $v$-generic polarization on~$A$. Then the rational Chow motive $\mathfrak{h}(\widetilde{K}_v(A,H))$ belongs to $\langle \mathfrak{h}(A)\rangle_{\mathsf{CHM}}$, the pseudo-Abelian tensor subcategory of motives generated by the motive of $A$.
		In fact, we have
		\begin{eqnarray*} \label{eq:formula}
		\h (\widetilde{K}_v (A,H)) & = & \h (\Km(A)^{(3)})  \oplus \h(\Km(A)\times \Km(A))(-1) \oplus \h(\Km(A\times A))(-1) \\ & & \oplus 137\h(\Km(A))(-2)\oplus 512\mathsf{Q}(-3).
		\end{eqnarray*}
   \end{theorem} 
    In the above formula, given an Abelian variety $B$, we denote by $\Km(B)$ the quotient~$B/\pm 1$; the rational Chow motive of $\Km(B)$ is defined as the motive $(B, \gamma, 0)$ with $\gamma=\frac{1}{2}(\delta+\delta^-)$, where $\delta$ (resp.\ $\delta^{-}$) is the diagonal (resp.\ the graph of $-1$) in~$B\times B$. 
    The motive of any (symmetric) power of $\Km(B)$ is defined similarly.  
    
    The proof of the theorem is based on the construction of Mongardi--Rapagnetta--Sacc\`a~\cite{MRS18} of $\widetilde{K}_v(A,H)$ as resolution of the quotient of a variety $\underline{Y}_{v}(A,H)$ of~K3$^{[3]}$-type by a symplectic birational involution, which was already used in \textit{loc.\ cit.\ }to calculate the Hodge numbers of OG$6$ varieties. The key point is then to show that $\underline{Y}_v(A,H)$ is motivated by $A$. To this end, we first observe that $\underline{Y}_v(A,H)$ is in fact birational to a moduli space of stable sheaves on a K3 surface $S$, and then we prove that $S$ is isogenous to the Kummer K3 surface obtained as minimal resolution of $\Km(A)$. 

	Theorem \ref{thm:main} has the following direct consequences.
	\begin{corollary}\label{cor:main}
		With $A$, $v$ and $H$ as in Theorem \ref{thm:main}, we have:
		\begin{itemize}
			\item the Chow motive of the OG$6$ variety $\widetilde{K}_v(A,H)$ is Abelian and finite dimensional;
			\item the conjectures of Hodge and Tate hold for $\widetilde{K}_v(A,H)$ and all of its self-products.
		\end{itemize}
	\end{corollary}

By the Tate conjecture for a smooth and projective complex variety $X$ we mean the Tate conjecture for any model $Y$ over a field $L\subset \CC$ finitely generated over $\QQ$ such that $Y\otimes_{L} \CC \cong X$; by \cite[\S1.6]{moonen2017}, its validity does not depend on the choice of the model $Y$.
In the last section we explain how our result gives some evidence towards the conjecture that isogenous hyper-K\"{a}hler varieties should have isomorphic motives.

	\subsection*{Notation and conventions}
	We work over the field of the complex numbers. We let $\mathsf{CHM}$ be the category of Chow motives with rational coefficients over~$\CC$ (\emph{cf.}~\cite{andre}). 
	The motive of a smooth and projective variety~$X$ is denoted by~$\mathfrak{h}(X)$. We denote by~$\mathsf{Q}$ the unit object for the tensor product in $\mathsf{CHM}$, and by $\mathsf{Q}(-1)$ the unique motive such that $\mathfrak{h}(\mathbb{P}^1)=\mathsf{Q}\oplus \mathsf{Q}(-1)$. 
	The dual of $\mathsf{Q}(-1)$ is denoted by~$\mathsf{Q}(1)$; given a motive $M$ and an integer $n>0$ we let $M(\pm n)$ denote $M\otimes \mathsf{Q}(\pm 1)^{\otimes n}$. 
	The (thick) pseudo-Abelian tensor subcategory of~$\mathsf{CHM}$ generated by a motive $M$ is denoted by~$\langle M \rangle_{\mathsf{CHM}} $: by definition, it is the full subcategory of $\mathsf{CHM}$ containing $M$ which is closed under isomorphisms, direct sums, tensor products, duals and subobjects.
	
    \subsection*{Aknowledgements}
    I am grateful to Lie Fu for many exchanges and for his comments on a first draft of this note. I thank Giuseppe Ancona and Domenico Valloni for useful discussions. I thank the anonymous referee for his/her comments.
	
	\section{The construction of Mongardi--Rapagnetta--Sacc\`a}
	
	In this section we describe the construction of the article \cite{MRS18}, to which we refer for proofs and more details. 
	Let $A,v,H$ be as in Theorem \ref{thm:main}. We denote by $\hat{A}$ the dual of the Abelian surface $A$. The singular projective moduli space $M_v(A,H)$ admits a smooth proper morphism
	\[
	a\colon M_v(A,H)\to A\times \hat{A}.
	\]
	In fact, $a$ is isotrivial (see \cite{Yos01}): denoting by $K_v(A,H)$ its fibre over any point, after an \'etale base-change $M_v(A,H)$ splits as the product of $K_v(A,H)$ with $A\times \hat{A}$. 	
	
 We briefly recall the geometry of the singular variety $K_{v}(A,H)$ and of its symplectic resolution $\widetilde{K}_v(A,H)$, following \cite{O'G03} and \cite{MRS18}. Let us now simply write $K$ for $K_v(A,H)$.
	We have closed subvarieties $\Omega \subset \Sigma \subset K$ such that
	\begin{itemize}
		\item $\Sigma\cong (A\times \hat{A})/\pm 1$ is the singular locus of $K$;
		\item $\Omega$ is the singular locus of $\Sigma$, and consists of $256$ points.
	\end{itemize}
	
	The resolution $\widetilde{K}\to K$ is constructed as follows.
	The blow-up of $K$ with center $\Omega$ yields a projective variety $\overline{K}\coloneqq \mathrm{Bl}_{\Omega}(K)$, with singular locus the strict transform $\overline{\Sigma}$ of $\Sigma$. 
	The blow-up of $\overline{K}$ with center $\overline{\Sigma}$ is a smooth projective variety $\widehat{K}\coloneqq \mathrm{Bl}_{\overline{\Sigma}}(\overline{K})$. Finally, the pre-image $\widehat{\Omega}\subset \widehat{K}$ of ${\Omega}$ can be contracted as a $\mathbb{P}^2$-bundle to $256$ disjoint smooth quadric threefolds, obtaining the OG$6$ variety~$\widetilde{K}$. 
	
	In fact, by \cite{LS06} the resolution $\widetilde{K}$ can be obtained via a single blow-up $\widetilde{K}=\mathrm{Bl}_{\Sigma}(K)$ along the (reduced) singular locus $\Sigma\subset K$. 
	We then have $\widehat{K}=\mathrm{Bl}_{\widetilde{\Omega}}(\widetilde{K})$ where $\widetilde{\Omega}\subset \widetilde{K}$ is the pre-image of $\Omega$.
	To summarize, we have a commutative diagram
	\begin{equation}\label{eq:diagramK}
	\begin{tikzcd}
		& \mathrm{Bl}_{\overline{\Sigma}}(\overline{K})=\widehat{K} =\mathrm{Bl}_{\widetilde{\Omega}}(\widetilde{K}) \arrow{ld} \arrow{rd} \\
		\overline{K}=\mathrm{Bl}_{\Omega}(K)  \arrow{dr} && \widetilde{K}=\mathrm{Bl}_{\Sigma}(K) \arrow{dl} \\
		& K
	\end{tikzcd}
	\end{equation}
where:
	\begin{itemize} 
	\item $\overline{\Sigma}\subset \overline{K}$ is smooth and isomorphic to $\mathrm{Bl}_{A[2]\times  \hat{A}[2]}\bigl((A\times \hat{A})/\pm 1\bigr)$;
	\item $\widetilde{\Omega}\subset \widetilde{K}$ is the disjoint union of $256$ smooth $3$-dimensional quadrics.
	\end{itemize}
	
	The main result of \cite{MRS18} gives  a normal projective variety $Y=Y_v(A,H)$ and a finite morphism $\epsilon\colon Y\to K$ of degree $2$. The branch locus of $\epsilon $ is $\Sigma$; the ramification locus $\Delta\subset Y$ is isomorphic to $\Sigma$. We let $\Gamma\subset \Delta$ be its singular locus, that is, $\Gamma=\epsilon^{-1}(\Omega)$ consists of~$256$ points. 
	We then have a commutative diagram of blow-ups
		\begin{equation}\label{eq:diagramY}
	\begin{tikzcd}
		& \mathrm{Bl}_{\overline{\Delta}}(\overline{Y})=\widehat{Y} =\mathrm{Bl}_{\widetilde{\Gamma}}(\widetilde{Y}) \arrow{ld} \arrow{rd} 
		\\
		\overline{Y}=\mathrm{Bl}_{\Gamma}(Y)  \arrow{dr} 
		&& \widetilde{Y}= \mathrm{Bl}_{\Delta}(Y) \arrow{dl} 
		\\
		& Y 
	\end{tikzcd}
	\end{equation}
	
	The morphism $\epsilon$ induces morphisms 
	 $\hat{\epsilon}\colon \widehat{Y} \to \widehat{K}$, $\overline{\epsilon}\colon \overline{Y} \to \overline{K}$, $\tilde{\epsilon}\colon \widetilde{Y}\to \widetilde{K}$ of degree~$2$, compatible with all blow-up maps in the diagrams \eqref{eq:diagramK} and \eqref{eq:diagramY}. In particular, these double covers are ramified over subvarieties denoted $\widehat{\Delta}$, $\overline{\Delta}$ and $\widetilde{\Delta}$ respectively, and they induce isomorphisms: $\hat{\epsilon}\colon \widehat{\Delta}\xrightarrow{\ \sim \ } \widehat{\Sigma}$, $\overline{\epsilon}\colon \overline{\Delta}\xrightarrow{\ \sim \ } \overline{\Sigma}$ and $\tilde{\epsilon}\colon \widetilde{\Delta}\xrightarrow{\ \sim \ }\widetilde{\Sigma}$.
	
	Moreover, the authors show that the variety $\overline{Y}$ is regular, and that the pre-image~$\overline{\Gamma}\subset \overline{Y}$ of $\Gamma$ can be contracted as a~$\mathbb{P}^2$-bundle to~$256$ copies of $\mathbb{P}^3$. The resulting variety $\underline{Y}$ is also regular, and it is a hyper-K\"{a}hler variety of $ \mathrm{K3}^{[3]}$-type. It contains~$256$ copies of $\mathbb{P}^3$; denoting their union by~$\underline{\Gamma}$, we have $\overline{Y}=\mathrm{Bl}_{\underline{\Gamma}}(\underline{Y})$.

   To summarize, for $A,v,H$ as above, we have a variety $\underline{Y}=\underline{Y}_v(A,H)$ of K$3^{[3]}$-type with a rational map of degree $2$ to the OG$6$-type variety $\widetilde{K}=\widetilde{K}_{v}(A,H)$, which is resolved via the following diagram:
   \begin{equation}\label{eq:diagram}
   	\begin{tikzcd}
   		\widehat{Y} = \mathrm{Bl}_{\overline{\Delta}}(\overline{Y}) \arrow{d} \arrow{drr}{\hat{\epsilon}}\\
   		\overline{Y}=\mathrm{Bl}_{\underline{\Gamma}}(\underline{Y}) \arrow{d} && \widehat{K}=\mathrm{Bl}_{\widetilde{\Omega}} (\widetilde{K}) \arrow{d} \\
   		\underline{Y} \arrow[dashed]{rr}   && \widetilde{K} 
   	\end{tikzcd}
   \end{equation}
   
   All varieties involved are smooth and projective. The morphism $\hat{\epsilon}$ is a degree $2$ cover ramified over the divisor $\widehat{\Sigma}$. The other maps are blow-ups with smooth centers:
   \begin{itemize}
   	\item $\widetilde{\Omega}$ consists of $256 $ disjoint smooth $3$-dimensional quadrics $G$;
   	\item $\underline{\Gamma}$ is a union of $256$ disjoint copies of $\mathbb{P}^3$;
   	\item $\overline{\Delta}\cong \mathrm{Bl}_{A[2]\times \hat{A}[2]}\bigl((A\times \hat{A})/\pm 1\bigr)$.
   \end{itemize}

	\section{The K3$^{[3]}$ variety as a moduli space}  \label{sec:modulispace}
	
	We denote by $H^2_{\tr}(X,\ZZ)\subset H^2(X,\ZZ)$ the transcendental part of the second Hodge structure of a smooth and projective variety $X$; by definition, $H^2_{\tr}(X,\ZZ)$ is the smallest $\ZZ$-Hodge substructure of $H^2(X,\ZZ)$ whose complexification contains $H^{2,0}(X)$.
	
	\begin{lemma}\label{lem:trascendental}
		Let $f\colon X \to Z $ be a birational map. Then the $\ZZ$-Hodge structures $H^2_{\tr}(X,\ZZ)$ and $H^2_{\tr}(Z,\ZZ)$ are isomorphic.
	\end{lemma}
\begin{proof}
	This is well-known: it is immediate when $f$ is a blow-up with smooth center, and the general case follows from this factoring $f$ as composition of blow-ups and blow-downs with smooth centers (\emph{cf.}~\cite{AKMW}).
\end{proof}

	When $X$ is a hyper-K\"{a}hler variety, $H^2_{\tr}(X,\ZZ)$ is the transcendental lattice of $X$: it is equipped with the restriction of the Beauville--Bogomolov form (\emph{cf.}~\cite{Huy99}). It has signature $(2, b_2-\rho-2)$ where $b_2$ and $\rho$ are the second Betti number and the Picard number of $X$ respectively. 
	Recall that the Beauville--Bogomolov form $q_X$ of a $2n$-dimensional hyper-K\"{a}hler manifold $X$ satisfies Fujiki's relation \cite{fujiki1987}: there exists a positive rational constant $c_X$ such that 
	\[
	\int_X \alpha^{2n} = c_X \cdot q_X(\alpha,\alpha)^n.
	\]
	
	Let $f\colon W\to X$ be a birational map with $X$ a hyper-K\"{a}hler variety of dimension~$2n$. By Lemma \ref{lem:trascendental} we have a Hodge isomorphism 
	\[
	H^2_{\tr}(W,\ZZ)\cong H^2_{\tr}(X,\ZZ).
	\]
	We then equip $H^2_{\tr}(W,\ZZ)$ with the form obtained from $H^2_{\tr}(X,\ZZ)$ via the above isomorphism and speak about the transcendental lattice of $W$.
	\begin{remark}\label{rmk:cfuji}
		The induced form $q_{W}$ on $H^2_{\tr}(W,\ZZ)$ satisfies Fujiki's relation
		\[
		\int_{W} \alpha^{2n} = c_X \cdot q_{W}(\alpha, \alpha)^{n},
		\]
		with the same Fujiki constant $c_X$ of $X$.
	\end{remark}

		Let $A,v,H$ be as in Theorem \ref{thm:main}. We use the notation from the previous section. 
		\begin{proposition}\label{prop:moduli}
			There exists a K3 surface $S$, a primitive and effective Mukai vector~$w$ and a $w$-generic polarization $L$ on $S$ such that the $\mathrm{K3}^{[3]}$-type variety $\underline{Y}_v(A,H)$ is birational to the smooth and projective moduli space $M_{w}(S,L)$.
		\end{proposition}
	
    We first prove a lemma. Let us consider the diagram \eqref{eq:diagram}.
	By Lemma \ref{lem:trascendental}, we have identifications of transcendental lattices: $H^2_{\tr}(\widehat{Y},\ZZ) \cong H^2_{\tr}(\underline{Y},\ZZ)$ and $			H^2_{\tr}(\widehat{K},\ZZ)\cong H^2_{\tr}(\widetilde{K},\ZZ)$.
Let $\hat{\iota}\colon \widehat{Y}\to \widehat{Y}$ be the involution associated with the double cover $\hat{\epsilon}$. The involution~$\hat{\iota}$ is symplectic, {\it i.e.} its action on the cohomology is the identity on $H^2_{\tr}(\widehat{Y}, \ZZ)$.

\begin{lemma}\label{lem:hodgeisometry}
	Via the isometries above, the morphism $\hat{\epsilon}\colon \widehat{Y} \to \widehat{K}$ induces a rational Hodge isometry
	\[
	\hat{\epsilon}_*\colon H^2_{\tr}(\underline{Y},\QQ)[2] \xrightarrow{\ \sim\ } H^2_{\tr}(\widetilde{K}, \QQ),
	\]
	where $[2]$ indicates that the form is multiplied by a factor $2$.
\end{lemma}
\begin{proof}
	For any $y\in H^2(\widehat{Y}, \QQ)^{\hat{\iota}}$, we have $\hat{\epsilon}^*\hat{\epsilon}_*(y)=2y$. Thus, $\epsilon_*$ is injective on~$H^2(\widehat{Y}, \QQ)^{\hat{\iota}}$ and hence it induces an isomorphism $\hat{\epsilon}_*\colon H_{\tr}^2(\widehat{Y}, \QQ)\cong H^2_{\tr}(\widehat{K}, \QQ)$ of Hodge structures. 
	For any $y \in H^2(\widehat{Y}, \QQ)^{\hat{\iota}}$, we then have (see Remark \ref{rmk:cfuji})
	\[
    q_{\widehat{K}}(\hat{\epsilon}_*(y), \hat{\epsilon}_*(y))^{3} 
	= \frac{1}{c_{\widetilde{K}}} \int_{\widehat{K}} \hat{\epsilon}_* (y)^{6}
	  = \frac{1}{2c_{\widetilde{K}}} \int_{\widehat{Y}} \hat{\epsilon}^*\hat{\epsilon}_*(y)^{6}
	= 2^5 \frac{c_{\underline{Y}}}{c_{\widetilde{K}}} \Bigl(\dfrac{1}{c_{\underline{Y}}} \int_{\hat{Y}} y^{6}\Bigr)
    = 2^5 \frac{c_{\underline{Y}}}{c_{\widetilde{K}}} \cdot q_{\widehat{Y}}(y,y)^3.
	\]
	The Fujiki constants of the known deformation types of hyper-K\"{a}hler manifolds can be found in \cite{rapagnetta2008beauville}. Since $\underline{Y}$ is of K3$^{[3]}$-type and $\widetilde{K}$ is of OG$6$-type we have $c_{\underline{Y}}=15$ and $c_{\widetilde{K}}=60$.
	We therefore obtain 
	$
	q_{\widetilde{K}}(\epsilon_* (y), \epsilon_*(y))=2 q_{\underline{Y}}(y,y)
	$, for any $y\in H^2_{\tr}(\underline{Y},\QQ)$.
\end{proof}

\begin{remark}\label{rem:H2}
In \cite{PR13} it is shown that there is a Hodge isometry $$H^2(K_{v}(A,H), \ZZ) \xrightarrow{\ \sim\ } v^{\bot} \subset \widetilde{H}^2(A,\ZZ),$$ and that $H^2(K_{v}(A,H), \ZZ)$ embeds isometrically into $H^2(\widetilde{K}_{v}(A,H), \ZZ)$ as Hodge substructure. It follows that we have a Hodge isometry $H^2_{\tr}(\widetilde{K}, \ZZ)\xrightarrow{\ \sim\ } H^2_{\tr}(A,\ZZ)$.
\end{remark}

\begin{proof}[Proof of Proposition \ref{prop:moduli}]
	According to \cite[Proposition 4]{Addington2016} (or equivalently \cite[Proposition~3.3]{mongardi2015}), it suffices to show that there exists a $\mathrm{K}3$ surface $S$ and a Hodge isometry $H_{\tr}^2(\underline{Y}, \ZZ)\cong H_{\tr}^2(S,\ZZ)$ of transcendental lattices. As the transcendental lattice of an Abelian surface has rank at most $5$, by Lemma~\ref{lem:hodgeisometry} and Remark~\ref{rem:H2} we obtain that $H^2_{\tr}(\underline{Y},\ZZ)$ is an  even lattice of rank~$r\leq 5$ and signature $(2, r-2)$. By~\cite[Corollary~2.10]{Morrison1984} any such lattice appears as transcendental lattice of some K3 surface.  
\end{proof}

	\section{Conclusion of the proof}
 
   Let $A,v,H$ be as in the statement of Theorem \ref{thm:main} and let $\underline{Y}=\underline{Y}_v(A,H)$ be the $\mathrm{K}3^{[3]}$-type variety associated to $\widetilde{K}=\widetilde{K}_v(A,H)$.
 
\begin{lemma}\label{lemma:easy}
	We have $\mathfrak{h}(\widetilde{K})\in \langle \mathfrak{h}(\underline{Y})\oplus \mathfrak{h}(A)\rangle_{\mathsf{CHM}}$.
\end{lemma} 
\begin{proof}
	We follow the same path as for the calculation of Hodge numbers \cite{MRS18} and Hodge structures \cite{green2019llv}. A similar argument is used in \cite[\S4.4]{fuLi2020}. 
	
	We consider the diagram \eqref{eq:diagram}. The motive of a smooth $3$-dimensional quadric is $\mathsf{Q} \oplus \mathsf{Q}(-1)\oplus \mathsf{Q}(-2) \oplus \mathsf{Q}(-3)$. Since $\overline{\Delta}=(\mathrm{Bl}_{A[2]\times \hat{A}[2]}(A\times \hat{A}))/\pm 1$, we have 
	\begin{align*}
		\mathfrak{h}(\overline{\Delta}) &= \Bigl(\mathfrak{h}(A\times \hat{A}) \oplus 256 \mathsf{Q}(-1) \oplus 256 \mathsf{Q}(-2) \oplus 256\mathsf{Q}(-3)\Bigr)^{\pm 1}\\
		&= \mathfrak{h}^+(A\times \hat{A})\oplus 256 \mathsf{Q}(-1) \oplus 256 \mathsf{Q}(-2) \oplus 256\mathsf{Q}(-3).
	\end{align*}
	Here, $\mathfrak{h}^+(A\times \hat{A})$ denotes the even part of the rational motive of $A\times \hat{A}$, which is well-defined since Abelian varieties have canonical Chow--K\"{u}nneth decompositions (see~\cite{deninger}).
	
	Applying the blow-up formula for motives (from \cite{Man68}), we obtain
	\begin{align*}
		\mathfrak{h}(\widehat{K})&=\mathfrak{h}(\widetilde{K}) \oplus 256 \mathsf{Q}(-1) \oplus 512 \mathsf{Q}(-2) \oplus 512 \mathsf{Q}(-3) \oplus 512 \mathsf{Q}(-4) \oplus 256 \mathsf{Q}(-5);
	\end{align*}
moreover, we have
	\begin{align*}
		\mathfrak{h}(\widehat{Y}) &= \mathfrak{h}(\overline{Y}) \oplus \mathfrak{h}^+(A\times \hat{A})(-1) \oplus 256 \mathsf{Q}(-2) \oplus 256 \mathsf{Q}(-3) \oplus 256\mathsf{Q}(-4),
	\end{align*}
	and
	\begin{align*}
		\mathfrak{h}(\overline{Y}) &= \mathfrak{h}(\underline{Y})\oplus 256 \mathsf{Q}(-1) \oplus 512 \mathsf{Q}(-2) \oplus 512 \mathsf{Q}(-3) \oplus 512 \mathsf{Q}(-4) \oplus 256 \mathsf{Q}(-5).
	\end{align*}
	
	The variety $\widehat{K}$ is the quotient of $\widehat{Y}$ by the regular involution $\hat{\iota}$ associated to the double cover $\hat{\epsilon}$. It follows that (since our motives are with rational coefficients) $$\mathfrak{h}(\widehat{K})\cong \mathfrak{h}(\widehat{Y})^{\hat{\iota}}$$
	belongs to $\langle \mathfrak{h}(\underline{Y}) \oplus \mathfrak{h}(A)\rangle_{\mathsf{CHM}}$; hence, the same is true for $\mathfrak{h}(\widetilde{K})$.
\end{proof}

\begin{proposition} \label{prop:last}
	The motive $\mathfrak{h}(\underline{Y})$ belongs to $\langle\mathfrak{h}(A)\rangle_{\mathsf{CHM}}$.
\end{proposition} 
\begin{proof}
	By Proposition \ref{prop:moduli}, the variety $\underline{Y}$ is birational to a moduli space $M_w(S,L)$ of stable sheaves on a $\mathrm{K}3$ surface $S$.
	By \cite{Rie14}, birational hyper-K\"{a}hler varieties have isomorphic Chow motives; B\"ulles' theorem in \cite{Bue18} then implies that $\mathfrak{h}(\underline{Y})\in \langle \mathfrak{h}(S)\rangle_{\mathsf{CHM}}$. By \cite{O'G97} we have a Hodge isometry 
	\[
    H^2(M_{w}(S, L), \ZZ) \xrightarrow{\ \sim\ } w^{\bot}\subset \widetilde{H}^2(S,\ZZ),
	\]
	which gives a Hodge isometry $H^2_{\tr}(\underline{Y},\ZZ)\xrightarrow{\ \sim\ } H^2_{\tr}(S,\ZZ)$ of transcendental lattices.
	
	Let now $T$ be the K3 surface which is the minimal resolution of $\Km(A)=A/\pm 1$. 
	It is well known (\emph{cf.}~\cite{morrison1985}) that $\mathfrak{h}(T)\in \langle \mathfrak{h}(A)\rangle_{\mathsf{CHM}}$ and that we have a Hodge isometry 
	\[
	H^2_{\tr}(T, \ZZ) \xrightarrow{\ \sim\ } H^2_{\tr}(A, \ZZ) [2].
	\]
	Combining this with Lemma \ref{lem:hodgeisometry} we obtain a rational Hodge isometry 
    \[
   \phi\colon  H^2_{\tr}(T ,\QQ)  \xrightarrow{\ \sim\ } H^2_{\tr}(A, \QQ)[2] \xrightarrow{\ \sim\ } H^2_{\tr}(S,\QQ)[4].
    \]	
    Then $\frac{1}{2} \phi \colon H^2_{\tr}(T ,\QQ) \xrightarrow{\ \sim\ } H^2_{\tr}(S,\QQ)$ is a Hodge isometry. 
	By Witt's theorem, we can extend it to an isometry  $\psi\colon H^2(T,\QQ)\xrightarrow{\ \sim\ } H^2(S, \QQ)$, which is necessarily compatible with the Hodge structures; hence, the K3 surfaces $S$ and $T$ are isogenous. 
	Now a result of Huybrechts \cite{huybrechtsMotives} implies that the motives of $S$ and $T$ are isomorphic, and we conclude that $\mathfrak{h}(\underline{Y})\in \langle\mathfrak{h}(A)\rangle_{\mathsf{CHM}}$. 
\end{proof}

	\begin{proof}[Proof of Theorem \ref{thm:main} and Corollary \ref{cor:main}]
	By Lemma \ref{lemma:easy} and Proposition \ref{prop:last}, the motive $\mathfrak{h}(\widetilde{K}_{v}(A,H))$ belongs to $\langle \h(A)\rangle_{\mathsf{CHM}}$ and hence, by definition, it is an Abelian motive. It is known that Abelian motives are Kimura finite dimensional (see \cite{Kimura}).
	
	By~\cite{moonenZarhin}, the Hodge conjecture holds for any self-product of the Abelian surface~$A$. Then~\cite[Lemma~4.2]{Ara06} implies that the Hodge conjecture holds for $\widetilde{K}_{v}(A,H)$ and any of its self-powers. Finally, by~\cite[Theorem~1.18]{FFZ} the Mumford--Tate conjecture holds for any projective variety of OG$6$-type; hence, the conjectures of Hodge and Tate are equivalent for such a variety and its self-products. This proves Corollary~\ref{cor:main}.
	
	It remains to prove the formula 
	\begin{align*} 
		\h (\widetilde{K}_v (A,H)) = & \h (\Km(A)^{(3)})  \oplus \h(\Km(A)\times \Km(A))(-1) \oplus \h(\Km(A\times A))(-1)  \\  & \oplus 137\h(\Km(A))(-2)\oplus 512\mathsf{Q}(-3).
		\end{align*}
		
	Let ${Z}$ denote the motive at the right hand side, and let $H^{\bullet}(Z)$ denote the associated rational Hodge structure. Comparing $H^{\bullet}(Z)$ with the description of the Hodge structure on the cohomology of $\widetilde{K}_v (A,H)$ given in \cite[\S3.5]{green2019llv}, we deduce that there is an isomorphism
	$\phi\colon H^{\bullet} (\widetilde{K}_{v}(A,H),\QQ) \xrightarrow{\ \sim\ } H^{\bullet}(Z)$ of rational Hodge structures. Since both $\h(\widetilde{K}_{v}(A,H))$ and $Z$ belong to $\langle \h(A) \rangle_{\mathsf{CHM}}$, the isomorphism $\phi$ is induced by the class of an algebraic cycle. Hence, there exists a morphism $\Phi\colon \h (\widetilde{K}_v (A,H)) \to Z$ of motives with Hodge realization $\phi$.
	Consider the realization functor $\mathsf{CHM} \to \mathsf{HS}$ to the category $\mathsf{HS}$ of rational Hodge structures; by \cite{AndreMotifs}, this functor is conservative when restricted to the tensor subcategory of Abelian motives. As the Hodge realization of $\Phi$ is an isomorphism in $\mathsf{HS}$, this implies that $\Phi$ is an isomorphism of Chow motives.
	\end{proof}
	
	\section{A conjecture}
	
	We say that two hyper-K\"{a}hler varieties $X$ and $Y$ are isogenous if they are deformation equivalent and there exists a rational Hodge isometry $H^2(X, \QQ)\xrightarrow{\ \sim \ } H^2(Y,\QQ)$. 	
	Theorem \ref{thm:main} gives some evidence towards the following.
	
	\begin{conjecture}\label{conj:isogenous} 
		Let $X$ and $Y$ be isogenous hyper-K\"{a}hler varieties. Then there exists an isomorphism of rational Chow motives $\h(X)\xrightarrow{\ \sim \ } \h(Y)$.
	\end{conjecture}
	For K3 surfaces the conjecture is proven in \cite{huybrechtsMotives}, and for Hilbert schemes on K3 surfaces it follows from the formula for their motives given in \cite{deCataldoMigliorini2002}. In the realm of Andr\'e motives, it has been verified in \cite{floccari2020} for all known deformation types and for hyper-K\"{a}hler varieties with trivial odd cohomology and $b_2>6$. 
	 
\begin{remark}
	A plausible stronger version of the above conjecture would require the existence of a multiplicative isomorphism $\h(X)\xrightarrow{\ \sim \ } \h(Y)$, meaning that it identifies the intersection product on the two sides. For K3 surfaces, this is proven in~\cite{FuVial}.
	
\end{remark}
	
	We point out that Conjecture \ref{conj:isogenous} would follow from two well-known conjectures: the Hodge conjecture, saying that the realization functor $r\colon \mathsf{CHM}\to \mathsf{HS}$ is full, and the conservativity conjecture, saying that the functor $r$ is conservative (see \cite{ayoub} for an inspiring overview of this conjecture). Indeed, isogenous hyper-K\"{a}hler varieties $X$ and~$Y$ have isomorphic rational Hodge structures (as shown in\cite{solda19}); assuming the Hodge conjecture, there exists a morphism $f\colon \h(X) \to \h(Y)$ of Chow motives whose realization is an isomorphism of Hodge structures, and hence, by conservativity, $f$ itself would be an isomorphism. 
	Besides $\mathrm{K}3$ surfaces, Conjecture \ref{conj:isogenous} is known in the following cases:
	\begin{enumerate}[label=(\roman*)]
		\item when $X$ and $Y$ are smooth and projective moduli spaces on K3 surfaces $S$ and $T$ whose motive is Kimura finite-dimensional, and such that the Hodge conjecture holds for any power of $S$ and $T$; for instance $S$ and $T$ may be Kummer surfaces (see \cite{Ped12} for more examples of K3 surfaces with finite-dimensional motive);
		\item when $X$ and $Y$ are ten dimensional O'Grady resolutions of moduli spaces on K3 surfaces $S$ and $T$ satisfying the same properties as above;
		\item for generalized Kummer varieties $X$ and $Y$ on Abelian surfaces $A$ and $B$;
		\item for OG$6$ varieties $X$ and $Y$ obtained via resolution of singular moduli spaces on Abelian surfaces $A$ and $B$.
	\end{enumerate}

Note that in each of the cases listed above, $X$ and $Y$ are isogenous if and only if the underlying surfaces are isogenous. 
The first two cases follow via the argument outlined above, because the realization functor $\mathsf{CHM}\to \mathsf{HS}$ is full and conservative when restricted to $\langle\h(X), \h(Y) \rangle_{\mathsf{CHM}}$. In fact, by \cite{Bue18} in case (i) and \cite{FFZ} in case (ii), we have the inclusion $\langle\h(X), \h(Y) \rangle_{\mathsf{CHM}} \subset \langle\h(S), \h(T) \rangle_{\mathsf{CHM}}$; the surfaces $S$ and $T$ are isogenous $\mathrm{K}3$ surfaces and hence they have isomorphic motives. Finally, the assumptions on $S$ and $T$ ensure that the realization functor is full and conservative when restricted to $\langle\h(S), \h(T) \rangle_{\mathsf{CHM}}$. 
In the last two cases, the validity of Conjecture~\ref{conj:isogenous} follows from the formulae for the motives of $X$ and $Y$ in terms of those of $A$ and $B$ given in \cite{Xu18}, \cite{FTV19} for case (iii) and by Theorem \ref{thm:main} for case (iv), as we observed that  $A$ and $B$ are isogenous and hence have isomorphic rational Chow motives.


\end{document}